\documentclass[11pt,twoside]{amsart}

\usepackage{prefix}
\newcommand{\mE}{\mathcal{E}}
\DeclareMathOperator{\supp}{supp}

\title{Completing shellings with $d-2$ extra vertices}

\author{Suho Oh}
\address{Department of Mathematics, Texas State University, 601 University Dr, San Marcos, TX 78666}
\email{suhooh@txstate.edu}

\begin{document}
\begin{abstract}
Simon's conjecture (1994) asserts that any pure $d$-dimensional shellable complex on $n$ vertices can be extended to the $d$-skeleton of the simplex on $n$ vertices, one facet at a time, while maintaining shellability. Bolognini and Sentinelli (2026) recently disproved it for every $d \geq 3$. We show that the conjecture becomes true once $d-2$ new vertices are allowed, and that the same holds with $k$-decomposability in place of shellability for any $k \geq 1$. We also show that the number $d-2$ cannot be lowered. Inflating the counterexample of Bolognini and Sentinelli gives, for every $d \geq 3$, a $1$-decomposable complex that cannot be extended in this way with only $d-3$ new vertices, even while merely maintaining shellability.
\end{abstract}
\maketitle

\section{Introduction}
\label{sec:intro}

A pure simplicial complex $\Delta$ is termed \emph{shellable} if one can order its facets $F_1, F_2, \dots, F_s$ so that each $F_i$ intersects the preceding facets in a pure complex of codimension one. Shellability serves as a vital combinatorial method with implications for both the topology of $\Delta$ and the algebraic properties of its Stanley--Reisner ring: a shellable complex is Cohen--Macaulay over every field, connected in codimension one, and has the homotopy type of a wedge of spheres. Notable instances of shellable simplicial complexes include independence complexes of matroids \cite{pb80}, boundary complexes of simplicial polytopes \cite{bm72}, and the skeleta of shellable complexes \cite{bw96}. Specifically, the $k$-skeleton of a simplex with vertex set $[n]$ is shellable for any $k = 1, 2, \dots, n-1$. Following \cite{bs26}, we denote it by $\Sk{k}{n}$.

When working with a shellable complex, one might naturally wonder if it is possible to encounter obstacles while constructing a shelling sequence. A shellable complex $\Delta$ is termed \emph{extendably shellable} if every shelling of a subcomplex of $\Delta$ can be expanded to form a shelling of $\Delta$ itself. In this context, a subcomplex of $\Delta$ refers to a simplicial complex whose facets form a subset of the facets of $\Delta$. Since every pure $d$-dimensional complex on $n$ vertices is a subcomplex of $\Sk{d}{n}$, the following question is central: which shellable complexes can be extended to $\Sk{d}{n}$ one facet at a time while maintaining shellability? Following \cite{cdgo22}, such complexes are called \emph{shelling completable}.

Shellable complexes are naturally found in various scenarios, but it appears that extendably shellable complexes are less prevalent. Some known examples are every $2$-dimensional triangulated sphere \cite{dk78}, rank $3$ matroid independence complexes \cite{be94} and all shellable $d$-complexes on $d+3$ vertices \cite{k77,cdgs20}. Björner and Eriksson \cite{be94} conjectured that the independence complex of every matroid is extendably shellable, but Hall \cite{hall04} found a counterexample. Ziegler \cite{z98} showed that some simplicial $4$-polytopes are not extendably shellable. In addition, there are $2$-dimensional complexes on $6$ vertices that are not extendably shellable \cite{mt03}. Hachimori \cite{Hachimori2000} gave an example of a shellable pure $2$-dimensional complex on $7$ vertices with a facet that must come last in any shelling order. Wachs \cite{w99} studied obstructions to shellability, and Adiprasito, Benedetti and Lutz \cite{AdiprasitoBenedettLutz2017} showed that for each $d \geq 2$ there exist shellable $d$-complexes with only one free face. The main motivation behind this paper and its predecessor \cite{ghkklotw25} is a question posed by Simon \cite{s94}. See also \cite[Ch.~III]{Stanley96} and \cite[Exercise 8.24(iii)]{z95}.

\begin{conjecture}[Simon's Conjecture \cite{s94}]
The complex $\Sk{d}{n}$ is extendably shellable.
\end{conjecture}

Some small cases have been resolved: the $d \leq 2$ case \cite{be94} and the $d \geq n-3$ case \cite{bpz19,d21,cdgs20}. Bigdeli, Yazdan Pour and Zaare-Nahandi \cite{bpz19} proposed a stronger conjecture implying Simon's, for which Benedetti and Bolognini \cite{bb21} found an infinite family of counterexamples. We refer the reader to \cite{dgl22,f25,e96} for related questions and conjectures. Stanley remarked that the conjecture is ``probably false'' \cite[p.~84]{Stanley96}, and indeed Bolognini and Sentinelli \cite{bs26} recently disproved it: for every $d \geq 3$ they constructed a pure $d$-dimensional shellable complex that is not shelling completable, the first being a $3$-dimensional complex on $16$ vertices with $280$ facets. Their counterexample is \emph{austere}: every missing facet contains two minimal nonfaces, so no facet at all can be added to it while maintaining shellability.

In light of this disproof, the question becomes how much room a shellable complex needs in order to be completable. There are two natural ways to relax Simon's conjecture: one can restrict the class of complexes, or one can allow new vertices, that is, ask for an extension to $\Sk{d}{m}$ for some $m > n$. Both were pursued in earlier work through the notion of $k$-decomposable complexes introduced by Provan and Billera \cite{pb80}, which interpolates between vertex decomposable complexes and shellable complexes: for $k=0$ it is vertex decomposability and for $k=d$ it is shellability. In \cite{cdgo22} it was shown that Simon's conjecture holds for $0$-decomposable complexes, with no new vertices needed.

\begin{theorem}[\cite{cdgo22}]
Given any pure $d$-dimensional $0$-decomposable simplicial complex $\Delta$ on $n$ vertices, we can extend $\Delta$ to $\Sk{d}{n}$ adding one facet at a time, while maintaining $0$-decomposability.
\end{theorem}

In \cite{ghkklotw25}, the analogous statement was shown for $1$-decomposable complexes at the cost of $d-2$ new vertices: any pure $d$-dimensional $1$-decomposable complex on $n$ vertices can be extended to $\Sk{d}{n+d-2}$ one facet at a time while maintaining $1$-decomposability. Two questions were left open there: whether the $d-2$ new vertices are really needed, and whether the same holds for $k$-decomposable complexes for larger $k$. In this paper we answer both. Our first main result is that $d-2$ new vertices suffice for every $k$, and in particular for all shellable complexes.

\theoremstyle{plain}
\newtheorem*{thmmain}{Theorem~\ref{thm:main}}
\begin{thmmain}
Let $d \geq 2$ and $k \geq 1$. Given any pure $d$-dimensional $k$-decomposable simplicial complex $\Delta$ on $n$ vertices, we can extend $\Delta$ to $\Sk{d}{n+d-2}$ one facet at a time, while maintaining $k$-decomposability.
\end{thmmain}

Taking $k = d$, every pure $d$-dimensional shellable complex on $n$ vertices can be extended to $\Sk{d}{n+d-2}$ one facet at a time while maintaining shellability. In other words, Simon's conjecture becomes true once $d-2$ extra vertices are allowed. Our second main result is that this is the exact price: the counterexample of Bolognini and Sentinelli, once each vertex of its $8$-vertex base is replaced by $d-1$ copies, shows that $d-3$ extra vertices are not enough even for $1$-decomposable complexes.

\newtheorem*{thmtight}{Theorem~\ref{thm:tight}}
\begin{thmtight}
For every $d \geq 3$ there is a pure $d$-dimensional $1$-decomposable complex $\mE$ on $8(d-1)$ vertices such that, for any set $A$ of at most $d-3$ new vertices, $\mE$ cannot be extended to the $d$-skeleton of the simplex on its vertex set together with $A$ one facet at a time while maintaining shellability.
\end{thmtight}

Hence for every $d \geq 3$ and $k \geq 1$, the complex $\Sk{d}{n+d-2}$ in \cref{thm:main} cannot be replaced by $\Sk{d}{n+d-3}$. For $d = 3$ this says that the $280$-facet counterexample of \cite{bs26} is itself $1$-decomposable, answering the first question of \cite{ghkklotw25} already in the smallest case.

The strategy for \cref{thm:main} follows \cite{ghkklotw25}: we cone the complex by $d-2$ new vertices, and use the cone vertices to enlarge the skeleton of the complex one edge or triangle at a time. The new ingredient that makes the argument work for all $k$, and considerably simplifies it, is that once the triangles have been filled in, the coned complex is automatically vertex decomposable (\cref{lem:conedfullVD}). The theorem of \cite{cdgo22} then finishes the job. The $1$-decomposability of the inflated counterexample is proved by lifting a $1$-decomposition of its $8$-vertex base.

The remainder of the paper is organized as follows. In Section~\ref{sec:prelim} we recall the necessary definitions, introduce the coning operation and $k$-full complexes, and prove that a sufficiently coned $k$-full complex is vertex decomposable. In Section~\ref{sec:main} we show how to add triangles to the base of a coned complex and prove \cref{thm:main}. In Section~\ref{sec:tight} we prove \cref{thm:tight}.

\section{Preliminaries}
\label{sec:prelim}

\subsection{Basic terminology}
We follow the conventions of \cite{ghkklotw25}. A \newword{simplicial complex} $\Delta$ on a finite ground set $V$ is a collection of subsets of $V$ that are closed under taking subsets. The elements of $\Delta$ are called \newword{faces}, and the elements $v \in V$ such that $\{v\} \in \Delta$ are called the \newword{vertices} of $\Delta$. Following \cite{j08}, the \newword{void complex} $\emptyset$ is distinct from the \newword{empty complex} $\{\emptyset\}$. A \newword{facet} of $\Delta$ is a face that is maximal under inclusion, the \newword{dimension} of $\Delta$ is the largest cardinality of a facet minus $1$, and $\Delta$ is \newword{pure} if all facets have the same cardinality. We use $\langle F_1,\dots,F_t \rangle$ to denote the simplicial complex with facets $F_1,\dots,F_t$. Given two facets $F_1$ and $F_2$ of a pure $d$-dimensional complex, we say that they are \newword{adjacent} if $|F_1 \cap F_2| = d$.

For a face $F \in \Delta$ we use $\lk_F{\Delta}$, $\plk_F{\Delta}$, and $\del_F{\Delta}$ for the \newword{link}, \newword{star}, and \newword{deletion} of $F$, defined as in \cite{ghkklotw25}:
\begin{align*}
\lk_F{\Delta} &:= \{G \in \Delta : G \cap F = \emptyset, G \cup F \in \Delta\}, \\
\plk_F{\Delta} &:= \{G \in \Delta : F \subset G\}, \\
\del_F{\Delta} &:= \{G \in \Delta: F \not \subset G\}.
\end{align*}
For a subset $W$ of the ground set, the \newword{restriction} of $\Delta$ to $W$ (also called the induced subcomplex on $W$) is $\restr{\Delta}{W} := \{G \in \Delta : G \subseteq W\}$. For a finite set $W$ we write $2^{W}$ for the simplex on $W$. For complexes $\Delta$ and $\Gamma$ on disjoint ground sets, their \newword{join} is $\Delta * \Gamma := \{G \cup G' : G \in \Delta,\ G' \in \Gamma\}$. For a finite set $W$ disjoint from the ground set of $\Delta$ we write $\Delta * W := \Delta * 2^{W}$ for the join with the simplex on $W$, and $\Delta * a$ when $W = \{a\}$, the \newword{cone} over $\Delta$ with apex $a$.

Given a simplicial complex $\Delta$ and an integer $t \geq -1$, we write $\skel{t}(\Delta)$ for the \newword{$t$-skeleton} of $\Delta$, the subcomplex consisting of all faces of $\Delta$ of dimension at most $t$. For a finite set $W$ we write $\Sk{d}{W} := \skel{d}(2^{W})$ for the $d$-skeleton of the simplex on $W$, and as in \cite{bs26} we abbreviate $\Sk{d}{[n]}$ to $\Sk{d}{n}$. We often treat $\skel{1}(\Delta)$ as a graph.

\begin{definition}[\cite{pb80}]
\label{def:kdec}
A pure $d$-dimensional simplicial complex $\Delta$ is said to be \newword{$k$-decomposable} if $\Delta$ is a simplex, or $\Delta$ contains a face $F$ such that
\begin{enumerate}
    \item $\dim(F) \leq k$,
    \item both $\del_F{\Delta}$ and $\lk_F{\Delta}$ are $k$-decomposable, and
    \item $\del_F{\Delta}$ is pure (and the dimensions stay the same as that of $\Delta$).
\end{enumerate}
A face $F$ that satisfies the first and third conditions is called a \newword{gluable face} of $\Delta$, and a face satisfying all three conditions is called a \newword{shedding face}.
\end{definition}

The $0$-decomposable complexes are the \newword{vertex decomposable} complexes, and $d$-decomposable coincides with shellable. Any pure $k$-decomposable complex is shellable \cite{pb80}.

\subsection{Combinatorial tools}
We recall the following tools from \cite{ghkklotw25}, which we will use throughout the paper. The first one is a purely combinatorial way to check if a face is a gluable face.

\begin{lemma}[Gluing criterion, {\cite[Lemma 2.5]{ghkklotw25}}]
\label{lem:shedd}
Suppose $\Delta$ is a pure simplicial complex, and let $F$ be a face with dimension $\leq k$. Then $F$ is a gluable face if and only if for any $f \in F$ and any facet $H$ of $\plk_F{\Delta}$, there exists a facet $H'$ in $\Delta$ such that $H \cap H' = H \setminus \{f\}$ (so they are adjacent).
\end{lemma}

As in \cite{ghkklotw25}, for pure complexes $\Delta$ and $\Gamma$ of the same dimension we write $\Delta + F$ for the complex generated by the facets of $\Delta$ together with $F$. We write $\Delta \setminus \Gamma$ for the complex generated by the facets of $\Delta$ that are not facets of $\Gamma$, and $\Delta + \Gamma$ for the complex generated by the facets of both when they share no facets.

\begin{lemma}[{\cite[Lemma 2.8]{ghkklotw25}}]
\label{lem:ordering}
Let $\Delta$ be a pure $k$-decomposable complex. Then we can order the facets $F_1,\dots,F_t$ such that $\langle F_1,\dots,F_i \rangle$ for $1 \leq i \leq t$ is $k$-decomposable.
\end{lemma}

\begin{lemma}[{\cite[Lemma 2.10]{ghkklotw25}}]
\label{lem:cheatcheckgluable}
Let $\Delta$ be a pure $k$-decomposable complex such that $F$ is either a gluable face or not a face of $\Delta$ at all. Let $F_1,\dots,F_t$ be facets not in $\Delta$ that do not contain $F$ and let $F_{t+1},\dots,F_q$ be facets not in $\Delta$ that contain $F$. If $\Delta + \langle F_1,\dots,F_q \rangle$ has $F$ as a gluable face, then $F$ is a gluable face in $\Delta + \langle F_1,\dots,F_i \rangle$ for any $1 \leq i \leq q$.
\end{lemma}

We end the subsection with a simple observation on skeleta that we will need.

\begin{lemma}
\label{lem:skelkdec}
Let $\Delta$ be a pure $d$-dimensional $k$-decomposable complex and let $0 \leq t \leq d$. Then $\skel{t}(\Delta)$ is $k$-decomposable.
\end{lemma}
\begin{proof}
We induct on the number of facets of $\Delta$, and may assume $t < d$. If $\Delta$ is a simplex, then $\skel{t}(\Delta)$ is a skeleton of a simplex, which is vertex decomposable \cite{pb80}. Otherwise let $F$ be a shedding face of $\Delta$. If $\dim(F) > t$, then no face of dimension at most $t$ contains $F$, so $\skel{t}(\Delta) = \skel{t}(\del_F{\Delta})$ is $k$-decomposable by induction.

So assume $\dim(F) \leq t$. We show that $F$ is a shedding face of $\skel{t}(\Delta)$. For the gluing criterion \cref{lem:shedd}, take $f \in F$ and a facet $H$ of $\plk_F{\skel{t}(\Delta)}$, and let $G$ be a facet of $\Delta$ containing $H$. As $F$ is gluable in $\Delta$, there is a facet $G'$ of $\Delta$ with $G \cap G' = G \setminus \{f\}$. Since $|G'| = d+1 > |H|$ we can pick $x \in G' \setminus H$. Then $H' := (H \setminus \{f\}) \cup \{x\}$ is a facet of $\skel{t}(\Delta)$ with $H \cap H' = H \setminus \{f\}$. Next, a face of $\skel{t}(\Delta)$ not containing $F$ is a face of $\del_F{\Delta}$ of dimension at most $t$, and $G \in \lk_F{\skel{t}(\Delta)}$ if and only if $G \cap F = \emptyset$, $G \cup F \in \Delta$ and $|G \cup F| \leq t+1$. Hence
\[
\del_F{\skel{t}(\Delta)} = \skel{t}(\del_F{\Delta}), \qquad \lk_F{\skel{t}(\Delta)} = \skel{t - |F|}(\lk_F{\Delta}),
\]
both $k$-decomposable by induction (when $|F| = t+1$ the link is the empty complex).
\end{proof}

When $k = 0$, the above lemma is \cite[Lemma 3.10]{Woodroofe2011}, and when $k = d$ it recovers the fact that skeleta of shellable complexes are shellable \cite{bw96}.

\begin{example}
Let
\[
\Delta = \langle 123,124,125,134,136,245,256,346,356,456 \rangle,
\]
which is $1$-decomposable with $15$ as a shedding face \cite{ghkklotw25}. Then $\skel{1}(\Delta)$ is the complete graph $K_6$, and $15$ is again a shedding face of it, with deletion $K_6$ minus an edge (a connected graph) and link $\{\emptyset\}$.
\end{example}

\subsection{Coning and fullness}
We now introduce the main operation of the paper. It is a dimension-preserving analogue of taking the cone.

\begin{definition}
\label{def:cone}
Let $\Gamma$ be a pure $d$-dimensional simplicial complex on $V$ and let $A$ be a finite set disjoint from $V$. The \newword{coning of $\Gamma$ by $A$} is the pure $d$-dimensional complex $\skel{d}(\Gamma * A)$ on $V \sqcup A$, whose facets are the sets $F \cup S$ with $F \in \Gamma$, $S \subseteq A$ and $|F| + |S| = d+1$. We say that a pure $d$-dimensional complex $\Delta$ on $V \sqcup A$ is \newword{coned with respect to $A$} if $\Delta = \skel{d}(\restr{\Delta}{V} * A)$, and call $\restr{\Delta}{V}$ its \newword{base}.
\end{definition}

For a single vertex $a$, the coning $\skel{d}(\Gamma * a) = \Gamma + \langle F \cup \{a\} : F \in \Gamma,\ |F| = d \rangle$ is obtained by attaching $a$ to every $(d-1)$-dimensional face of $\Gamma$, and coning by $A$ is coning by the elements of $A$ one at a time. We use the notation $\skel{d}(\Gamma * A)$ also when $\Gamma$ is not pure or has dimension less than $d$ (for instance a graph with isolated vertices). When $\Gamma = G$ is a graph, $\skel{d}(G * A)$ is the complex $\full^{d}(G^{A})$ of \cite{ghkklotw25}.

\begin{example}
Take $\Gamma = \langle 123,134 \rangle$ and $a = 5$. The $1$-dimensional faces of $\Gamma$ are $12,13,14,23,34$, so
\[
\skel{2}(\Gamma * 5) = \langle 123,134,125,135,145,235,345 \rangle.
\]
\end{example}

Coning preserves $k$-decomposability, and it does so adding one facet at a time.

\begin{lemma}
\label{lem:conekdec}
Let $\Gamma$ be a pure $d$-dimensional $k$-decomposable complex on $V$ and let $A$ be a finite set disjoint from $V$. Then we can order the facets of $\skel{d}(\Gamma * A) \setminus \Gamma$ as $F_1,\dots,F_q$ so that $\Gamma + \langle F_1,\dots,F_i \rangle$ is $k$-decomposable for all $1 \leq i \leq q$. In particular $\skel{d}(\Gamma * A)$ is $k$-decomposable.
\end{lemma}
\begin{proof}
Coning by $A$ is coning by one vertex at a time, so we may assume $A = \{a\}$. By \cref{lem:skelkdec,lem:ordering} we can order the facets of $\skel{d-1}(\Gamma)$ as $G_1,\dots,G_q$ so that each $\langle G_1,\dots,G_i \rangle$ is $k$-decomposable, and we set $F_i := G_i \cup \{a\}$, so that $\skel{d}(\Gamma * a) = \Gamma + \langle F_1,\dots,F_q \rangle$. Let $\Delta_i := \Gamma + \langle F_1,\dots,F_i \rangle$. Each facet of $\plk_a{\Delta_i}$ has the form $G_j \cup \{a\}$, and $G_j$ lies in a facet of $\Gamma$ adjacent to it, so $a$ is gluable in $\Delta_i$ by \cref{lem:shedd}. As $\del_a{\Delta_i} = \Gamma$ and $\lk_a{\Delta_i} = \langle G_1,\dots,G_i \rangle$ are $k$-decomposable, so is $\Delta_i$.
\end{proof}

\begin{example}
For $\Gamma = \langle 123,134 \rangle$ as above, ordering the edges as $12$, $13$, $23$, $14$, $34$ (each prefix is a connected graph) and adding
\[
125,135,235,145,345
\]
in this order keeps the complex vertex decomposable at every step.
\end{example}

\begin{definition}
\label{def:kfull}
Let $k \geq -1$. A simplicial complex $\Delta$ on $V$ is \newword{$k$-full} if every subset of $V$ of cardinality at most $k+1$ is a face of $\Delta$, that is, $\skel{k}(\Delta)$ is the $k$-skeleton of the simplex on $V$.
\end{definition}

Every complex is $(-1)$-full, a complex is $0$-full if every element of its ground set is a vertex, and $1$-full if $\skel{1}(\Delta)$ is complete. If $\Delta$ is $k$-full then so is $\restr{\Delta}{W}$ for any $W \subseteq V$. The first main tool of the paper is that a coned complex with enough cone vertices, compared to how full it is, is automatically vertex decomposable.

\begin{lemma}
\label{lem:conedfullVD}
Let $k \geq -1$ and let $\Delta$ be a pure $d$-dimensional simplicial complex on $V \sqcup A$ that is coned with respect to $A$ and is $k$-full. If $|A| \geq d - k$, then $\Delta$ is vertex decomposable.
\end{lemma}
\begin{proof}
Let $\Gamma := \restr{\Delta}{V}$, so that $\Delta = \skel{d}(\Gamma * A)$ and the facets of $\Delta$ are the sets $F \cup S$ with $F \in \Gamma$, $S \subseteq A$ and $|F| + |S| = d+1$. We induct on $|V|$, for all $k$, $d$ and $A$ at once, and may assume that every element of $V$ is a vertex of $\Delta$, since removing the other elements from $V$ changes nothing. If $\Delta$ is a simplex there is nothing to prove, and if $V = \emptyset$ then $\Delta = \skel{d}(2^{A})$ is a skeleton of a simplex, which is vertex decomposable \cite{pb80}. So assume $\Delta$ is not a simplex and take any $v \in V$. We show that $v$ is a shedding vertex of $\Delta$.

We first check that $v$ is gluable. Take a facet $F \cup S$ of $\Delta$ with $v \in F$. If $S \neq A$, pick $a \in A \setminus S$. Then $(F \setminus \{v\}) \cup (S \cup \{a\})$ is a facet of $\Delta$ adjacent to $F \cup S$, as $F \setminus \{v\} \in \Gamma$. If $S = A$, then $|F| = d+1-|A| \leq k+1$, and $V \not\subseteq F$ since $\Delta$ is not a simplex, so we can pick $u \in V \setminus F$. The set $(F \setminus \{v\}) \cup \{u\}$ has cardinality at most $k+1$ and lies in $V$, so it is a face of $\Gamma$ by $k$-fullness, and $(F \setminus \{v\}) \cup \{u\} \cup A$ is a facet of $\Delta$ adjacent to $F \cup S$ avoiding $v$.

Next we look at the deletion. The facets of $\del_v{\Delta}$ are the facets $F \cup S$ with $v \notin F$, so $\del_v{\Delta} = \skel{d}(\del_v{\Gamma} * A)$ is pure of dimension $d$, coned with respect to $A$, and still $k$-full. By induction it is vertex decomposable.

Finally we look at the link. The facets of $\lk_v{\Delta}$ are the sets $(F \setminus \{v\}) \cup S$ with $v \in F$, so $\lk_v{\Delta} = \skel{d-1}(\lk_v{\Gamma} * A)$ is pure of dimension $d-1$ and coned with respect to $A$ on $V' \sqcup A$, where $V'$ is the vertex set of $\lk_v{\Gamma}$. If $k \geq 0$, it is $(k-1)$-full, since $T \cup \{v\}$ is a face of $\Delta$ whenever $|T| \leq k$, and $|A| \geq d-k = (d-1)-(k-1)$. If $k = -1$, then $|A| \geq d+1 > (d-1)+1$. Either way $\lk_v{\Delta}$ is vertex decomposable by induction.
\end{proof}

The case $k = 0$ says that $\skel{d}(G * A)$ is vertex decomposable for any graph $G$ whenever $|A| \geq d$, which strengthens \cite[Lemma 4.3]{ghkklotw25} from $1$-decomposable to vertex decomposable.

\begin{example}
Take $G$ on vertex set $[5]$ with edges $12$ and $34$ (so $5$ is an isolated vertex), and let $A = \{a,b\}$. Then
\[
\skel{2}(G * A) = \langle ab1,ab2,ab3,ab4,ab5,a12,b12,a34,b34 \rangle
\]
is coned with respect to $A$, is $0$-full, and $|A| = 2 \geq 2 - 0$, so it is vertex decomposable by \cref{lem:conedfullVD}. Indeed any vertex of $[5]$ is a shedding vertex.
\end{example}

\section{Main result}
\label{sec:main}

\subsection{Adding triangles to the base}
\label{subsec:triangles}
In this subsection we grow the base of a coned complex while maintaining $k$-decomposability. Throughout, $\Delta$ is a pure $d$-dimensional complex on $V \sqcup A$ coned with respect to $A$ with base $\Gamma := \restr{\Delta}{V}$. The faces of $\Delta$ are then the sets $F \cup S$ with $F \in \Gamma$, $S \subseteq A$ and $|F| + |S| \leq d+1$. For $F \subseteq V$ not a face of $\Gamma$, \newword{adding $F$ to the base} means replacing $\Gamma$ by $\Gamma \cup 2^{F}$ and $\Delta$ by $\skel{d}((\Gamma \cup 2^{F}) * A)$. The new faces of cardinality $d+1$ are the sets $F' \cup S$ with $F' \subseteq F$ not a face of $\Gamma$, $S \subseteq A$ and $|F'| + |S| = d+1$. We will always have $|F| + |A| \geq d+1$, in which case the new complex is again pure.

\begin{lemma}
\label{lem:addtriangle}
Let $d \geq 2$ and $k \geq 2$. Let $\Delta$ be a pure $d$-dimensional $k$-decomposable complex on $V \sqcup A$, coned with respect to $A$ with base $\Gamma$, where $|A| \geq d-2$. Let $uvw \subseteq V$ be a triangle that is not a face of $\Gamma$ such that at least two of the edges $uv,vw,uw$ are faces of $\Gamma$, and let $\Delta'$ be the complex obtained by adding $uvw$ to the base. Then we can order the facets of $\Delta' \setminus \Delta$ as $F_1,\dots,F_q$ so that $\Delta + \langle F_1,\dots,F_i \rangle$ is $k$-decomposable for all $1 \leq i \leq q$. In particular, $\Delta'$ is $k$-decomposable and coned with respect to $A$.
\end{lemma}
\begin{proof}
Let $F := uvw$ if all three edges are in $\Gamma$, and otherwise let $F$ be the missing edge, say $F = uw$ with $uv, vw \in \Gamma$. In both cases $\dim(F) \leq k$ and $F$ is not a face of $\Delta$. The facets of $\Delta' \setminus \Delta$ are the sets $uvw \cup S$ with $S \subseteq A$ and $|S| = d-2$, together with the sets $uw \cup S$ with $|S| = d-1$ in the second case (these exist only when $|A| \geq d-1$). Every one of them contains $F$.

We first show that $F$ is gluable in $\Delta'$. The facets of $\Delta'$ containing $F$ are those of $\Delta' \setminus \Delta$. Let $G$ be such a facet and $f \in F$. Since $uvw \setminus \{f\}$ is an edge of $\Gamma$, the set $G \setminus \{f\}$ is a face of $\Delta$. Let $G'$ be a facet of $\Delta$ containing it. Then $G \cap G' = G \setminus \{f\}$, as $F$ is not a face of $\Delta$, so $F$ is gluable by \cref{lem:shedd}. Next, $\del_F{\Delta'} = \Delta$, and the link of $F$ in $\Delta'$ is $\skel{d-3}(2^{A})$ in the first case and $\skel{d-2}(2^{\{v\} \cup A})$ in the second, a skeleton of a simplex in either case and hence vertex decomposable \cite{pb80}.

Now order the facets of $\lk_F{\Delta'}$ as $G_1,\dots,G_q$ with each $\langle G_1,\dots,G_i \rangle$ $k$-decomposable, using \cref{lem:ordering}, and set $F_i := G_i \cup F$. By \cref{lem:cheatcheckgluable}, $F$ is gluable in $\Delta + \langle F_1,\dots,F_i \rangle$ for each $i$, with deletion $\Delta$ and link $\langle G_1,\dots,G_i \rangle$, both $k$-decomposable. Hence each $\Delta + \langle F_1,\dots,F_i \rangle$ is $k$-decomposable, and $\Delta'$ is coned with respect to $A$ by construction.
\end{proof}

Since the conclusion of \cref{lem:addtriangle} has the same form as its hypothesis, it can be applied repeatedly.

\begin{corollary}
\label{cor:addtriangles}
Let $d \geq 2$, $k \geq 2$, and let $\Delta$ be a pure $d$-dimensional $k$-decomposable complex coned with respect to $A$ with $|A| \geq d-2$. Let $\restr{\Delta}{V} = \Gamma_1 \subseteq \Gamma_2 \subseteq \dots \subseteq \Gamma_m$ be a sequence of complexes on $V$ where each $\Gamma_{j}$ with $j \geq 2$ is obtained from $\Gamma_{j-1}$ by adding a missing triangle that has at least two edges in $\Gamma_{j-1}$. Then we can extend $\Delta$ to $\skel{d}(\Gamma_m * A)$ one facet at a time while maintaining $k$-decomposability.
\end{corollary}

\begin{example}
Let $d = 3$ and $A = \{a\}$, and let the base be the $2$-dimensional complex $\Gamma = \langle 123, 234 \rangle$. Then $\Delta = \skel{d}(\Gamma * a) = \langle 123a, 234a \rangle$ is vertex decomposable. The triangle $124$ has exactly two edges $12,24$ in $\Gamma$ and $14 \not\in \Gamma$. Adding $124$ to the base means adding the single facet $124a$, using $14$ as the shedding face: we have $\lk_{14}(\Delta + 124a) = \langle 2a \rangle$ and $\del_{14}(\Delta + 124a) = \Delta$. Now $134$ has all three edges $13,34,14$ in the new base $\Gamma \cup 2^{124}$, so we can add the facet $134a$ using $134$ as the shedding face, whose link is $\langle a \rangle$. We end up with $\langle 123a,234a,124a,134a \rangle$, the cone over the boundary of the tetrahedron $1234$.
\end{example}

The example shows the difference between the two cases. When $|A| = d-2$, adding a missing edge alone contributes no new facets, so an edge can only enter the base as part of a triangle, with the edge as the shedding face. This is the mechanism behind \cite[Lemma 4.5]{ghkklotw25}. A triangle whose three edges are already present requires the triangle itself as the shedding face, hence $k \geq 2$.

\subsection{Proof of the main theorem}
\label{subsec:mainproof}
We need two results from the literature.

\begin{lemma}[{\cite[Section 7.2]{bj92}}]
\label{lem:skelconn}
Let $\Delta$ be a pure shellable complex of dimension at least $1$. Then $\skel{1}(\Delta)$ is a connected graph.
\end{lemma}

\begin{theorem}[{\cite[Theorem 2.9 and Corollary 2.11]{cdgo22}}]
\label{thm:cdgo}
Let $\Delta$ be a pure $d$-dimensional vertex decomposable complex with vertex set $[m]$. Then we can order the facets of $\Sk{d}{m} \setminus \Delta = \{F_1,\dots,F_t\}$ so that $\Delta + \langle F_1,\dots,F_i \rangle$ is vertex decomposable for all $1 \leq i \leq t$.
\end{theorem}

We now prove the main result of the paper.

\begin{theorem}
\label{thm:main}
Let $d \geq 2$ and $k \geq 1$, and let $\Delta$ be a pure $d$-dimensional $k$-decomposable complex with vertex set $[n]$. Then we can order the facets of $\Sk{d}{n+d-2} \setminus \Delta = \{F_1,\dots,F_t\}$ so that $\Delta + \langle F_1,\dots,F_i \rangle$ is $k$-decomposable for all $1 \leq i \leq t$.
\end{theorem}
\begin{proof}
The case $k = 1$ is \cite[Theorem 4.10]{ghkklotw25}, so assume $k \geq 2$. Let $A$ be a set of $d-2$ vertices disjoint from $[n]$, so that $\Sk{d}{n+d-2}$ is the $d$-skeleton of the simplex on $[n] \sqcup A$. Throughout, we add one facet at a time.

We first use \cref{lem:conekdec} to extend $\Delta$ to $\skel{d}(\Delta * A)$ while maintaining $k$-decomposability. This complex is coned with respect to $A$ with base $\Delta$. Next we add triangles to the base using \cref{cor:addtriangles}. Since $\Delta$ is shellable, $\skel{1}(\Delta)$ is connected by \cref{lem:skelconn}. While the $1$-skeleton of the current complex is not complete, there are $u,w \in [n]$ at distance $2$ with a common neighbor $v$, so the triangle $uvw$ is not a face and has exactly two edges in the base. Adding it adds the edge $uw$ to the skeleton. Once the skeleton is complete, every missing triangle on $[n]$ has all three edges in the base, and as $k \geq 2$ we add them one at a time. We end with $\Delta' := \skel{d}(\Gamma * A)$, where $\Gamma \supseteq \Delta$ is $2$-full.

Now $\Delta'$ is coned with respect to $A$ and $2$-full. As $|A| = d-2$, \cref{lem:conedfullVD} shows that $\Delta'$ is vertex decomposable. Here, the vertex set of $\Delta'$ is $[n] \sqcup A$ and has cardinality $n+d-2$. By \cref{thm:cdgo} we extend $\Delta'$ to $\Sk{d}{n+d-2}$ while maintaining vertex decomposability, and in particular $k$-decomposability.
\end{proof}

\begin{remark}
When $d = 2$ we have $A = \emptyset$, and \cref{thm:main} says that any shellable $2$-dimensional complex on $[n]$ can be extended to $\Sk{2}{n}$ while maintaining shellability, recovering the $d = 2$ case of Simon's conjecture \cite{be94}. Here adding triangles already reaches $\Sk{2}{n}$.
\end{remark}

\section{The number of new vertices is tight}
\label{sec:tight}

In this section we show that the $d-2$ new vertices in \cref{thm:main} cannot be reduced to $d-3$, for any $d \geq 3$ and any $k \geq 1$. The examples come from the construction that Bolognini and Sentinelli \cite{bs26} used to disprove Simon's conjecture.

\subsection{An obstruction to adding a facet}
\label{subsec:obstruction}
A \newword{minimal nonface} of a complex $\Delta$ on ground set $V$ is a subset of $V$ that is not a face of $\Delta$ but all of whose proper subsets are. The following observation is the reason austere complexes in the sense of \cite{bs26} cannot be extended. We state it in our facet extension language: the new facet is excluded no matter where it would appear in a shelling.

\begin{lemma}
\label{lem:twononfaces}
Let $\Delta$ be a pure $d$-dimensional complex and let $F$ be a $(d+1)$-subset of the ground set that is not a face of $\Delta$. If $F$ contains two distinct minimal nonfaces of $\Delta$, then $\Delta + F$ is not shellable.
\end{lemma}
\begin{proof}
We induct on $d$. The case $d = 0$ is vacuous, since two distinct minimal nonfaces are nonempty and incomparable. For $y \in F$, the $d$-subset $F \setminus \{y\}$ is a face of $\Delta$ if and only if it contains no minimal nonface of $\Delta$, that is, if and only if $y$ lies in every minimal nonface of $\Delta$ contained in $F$.

If no $y \in F$ does, then $F$ is adjacent to no facet of $\Delta$, so $F$ is an isolated facet of $\Delta + F$, which has other facets and positive dimension. This contradicts shellability, since in a shelling each facet after the first is adjacent to an earlier one.

Otherwise pick $y \in F$ lying in every minimal nonface of $\Delta$ contained in $F$, and let $N_1, N_2$ be two of them. Then $F \setminus \{y\}$ is a $d$-subset of the ground set containing the distinct minimal nonfaces $N_1 \setminus \{y\}$, $N_2 \setminus \{y\}$ of $\lk_y{\Delta}$. As $\lk_y{\Delta}$ is pure of dimension $d-1$, by induction $\lk_y{\Delta} + (F \setminus \{y\}) = \lk_y(\Delta + F)$ is not shellable, and hence neither is $\Delta + F$, since links of shellable complexes are shellable \cite[Lemma 8.7]{z95}.
\end{proof}

\subsection{Inflating a quiet complex}
\label{subsec:inflate}
Recall from \cite{bs26} that a pure $2$-dimensional complex $\Gamma$ on $[m]$ is \newword{quiet} if it has a complete $1$-skeleton and every $4$-subset of $[m]$ contains at most two facets of $\Gamma$. Bolognini and Sentinelli found a quiet, shellable, contractible complex $\Gamma_0$ on $[8]$ with $21$ facets:
\[
\begin{aligned}[t]
\Gamma_0 = \langle &148, 348, 234, 347, 368, 356, 168, 357, 236, 147, 167, \\
&156, 267, 257, 245, 125, 458, 128, 278, 123, 456 \rangle.
\end{aligned}
\]

The following construction is used in \cite{bs26} with classes of two elements, under the name echo, but is not formally defined there in this generality.

\begin{definition}
\label{def:inflate}
Let $\Gamma$ be a simplicial complex on a finite set $M$ and let $\C = (C_i)_{i \in M}$ be a family of pairwise disjoint finite sets, the \newword{classes}. Let $V := \bigsqcup_{i \in M} C_i$. For $F \subseteq V$, the \newword{support} of $F$ is $\supp(F) := \{i \in M : F \cap C_i \neq \emptyset\}$. The \newword{inflation} of $\Gamma$ with respect to $\C$ is the complex
\[
\mE^{\C}(\Gamma) := \{F \subseteq V : \supp(F) \in \Gamma\}
\]
on $V$, and for $t \geq 0$ the \newword{$t$-dimensional inflation} is $\mE^{\C}_t(\Gamma) := \skel{t}(\mE^{\C}(\Gamma))$. Throughout this section $d \geq 3$ is fixed, and we usually omit $\C$ from the notation, in which case every class has cardinality $d-1$.
\end{definition}

The faces of $\mE^{\C}(\Gamma)$ are the subsets of the sets $\bigcup_{i \in \sigma} C_i$ with $\sigma \in \Gamma$, so $\mE^{\C}(\Gamma)$ is obtained from $\Gamma$ by replacing each vertex $i$ with the simplex $2^{C_i}$ (a vertex with an empty class disappears), and $\mE^{\C}_t(\Gamma)$ consists of those faces of cardinality at most $t+1$.

\begin{example}
Let $d = 4$, so each class has three elements, say $C_i = \{3i-2, 3i-1, 3i\}$, and let $\Gamma = \langle 123, 234 \rangle$. Then $\mE_4(\Gamma)$ is generated by the $5$-subsets of $\{1,\dots,9\}$ together with the $5$-subsets of $\{4,\dots,12\}$, which gives $2\binom{9}{5} - \binom{6}{5} = 246$ facets.
\end{example}

When $d = 3$, each $C_i$ has two elements and $\mE_3(\Gamma)$ is exactly the \newword{echo} of $\Gamma$ from \cite[Definition 3.9]{bs26}. In particular $\mE_3(\Gamma_0)$ is the $3$-dimensional complex on $16$ vertices with $280$ facets that disproves Simon's conjecture.

The following lemma and its proof are due to Bolognini and Sentinelli \cite[Theorem 3.11]{bs26} for $d = 3$, and the same argument works for every $d$.

\begin{lemma}
\label{lem:inflate}
Let $\Gamma$ be a quiet complex on $[m]$ and $d \geq 3$. Then $\mE_d(\Gamma)$ is pure of dimension $d$, and:
\begin{enumerate}
    \item the minimal nonfaces of $\mE_d(\Gamma)$ of cardinality at most $d+1$ are exactly the sets $pqr$ with $p \in C_a$, $q \in C_b$, $r \in C_c$ where $abc$ is not a face of $\Gamma$, and
    \item if $N \subseteq V$ with $4 \leq |N| \leq d+1$ is not a face of $\mE_d(\Gamma)$, then $N$ contains two distinct minimal nonfaces of $\mE_d(\Gamma)$.
\end{enumerate}
\end{lemma}
\begin{proof}
Since $\Gamma$ is pure with a complete $1$-skeleton, the support of any face of $\mE_d(\Gamma)$ lies in some facet $abc$ of $\Gamma$, and $|C_a \cup C_b \cup C_c| = 3d-3 \geq d+1$, so every face extends to a facet. Hence $\mE_d(\Gamma)$ is pure of dimension $d$.

For (1), let $N$ be a minimal nonface with $|N| \leq d+1$ and let $\sigma := \supp(N)$. Then $\sigma$ is not a face of $\Gamma$, so $|\sigma| \geq 3$ as the $1$-skeleton of $\Gamma$ is complete. By minimality, $N$ meets each class of $\sigma$ in exactly one element, since removing a second element from a class would leave a proper subset of $N$ with the same support. Suppose $|\sigma| \geq 4$. As $N$ is a minimal nonface and $\Gamma$ is $2$-dimensional, this forces $|\sigma| = 4$ with all four triples in $\sigma$ facets of $\Gamma$, which goes against $\Gamma$ being quiet. Hence $|\sigma| = 3$ and $N$ has the stated form. Conversely, such a set $pqr$ is a nonface whose proper subsets have supports of cardinality at most $2$, hence are faces.

For (2), as above $|\supp(N)| \geq 3$. If $|\supp(N)| = 3$, then since $|N| \geq 4$ some class contains two elements $p \neq p'$ of $N$, and picking $q,r \in N$ from the other two classes, both $pqr$ and $p'qr$ are minimal nonfaces by (1). If $|\supp(N)| \geq 4$, then $N$ contains elements $p,q,r,s$ from four different classes. By quietness, at least two of the four triples of these classes are not faces of $\Gamma$, and the corresponding triples among $p,q,r,s$ are two distinct minimal nonfaces by (1).
\end{proof}

The second item says slightly more than austerity: not only the missing facets, but every missing face with at least four vertices, contains two minimal nonfaces.

\subsection{Lifting a decomposition of the base}
\label{subsec:lift}
We now show that $\mE_d(\Gamma_0)$ is $1$-decomposable for every $d \geq 3$. The idea of lifting structure from the base to the inflation one facet of $\Gamma$ at a time is from \cite[Theorem 3.12]{bs26}, where a shelling of $\Gamma$ is lifted to a shelling of the echo. We lift a $1$-decomposition instead, replacing each shedding face of the base by a bundle of shedding faces of the inflation. This needs the shedding faces of the base to obey a stricter condition.

\begin{definition}
\label{def:simply}
We call a pure $2$-dimensional complex $\Gamma$ \newword{simply $1$-decomposable} if $\Gamma$ is a simplex, or $\Gamma$ has a shedding face $F$ such that $\del_F{\Gamma}$ is simply $1$-decomposable and $F$ is either an edge contained in exactly one facet of $\Gamma$ or a vertex whose link is a tree.
\end{definition}

Since the link of a shedding vertex of a pure $2$-dimensional complex is a connected graph, the second condition only excludes cycles in the link.

\begin{example}
\label{ex:gamma0simply}
The complex $\Gamma_0$ is simply $1$-decomposable. Shedding the edges
\[
13, 46, 58, 45, 24, 23, 26
\]
in this order, each contained in exactly one facet of the current complex, removes the facets
\[
123, 456, 458, 245, 234, 236, 267
\]
and leaves the complex
\[
\langle 148, 348, 347, 368, 356, 168, 357, 147, 167, 156, 257, 125, 128, 278 \rangle.
\]
We then shed the vertices $7$, $2$, $1$, $4$, $5$ in this order. The link of each vertex in the current complex, and the complex that remains after shedding it, are as follows.
\[
\begin{array}{c|l|l}
\text{vertex} & \text{link} & \text{remaining complex} \\ \hline
7 & \langle 16,14,34,35,25,28 \rangle & \langle 148,348,368,356,168,156,125,128 \rangle \\
2 & \langle 15,18 \rangle & \langle 148,348,368,356,168,156 \rangle \\
1 & \langle 48,68,56 \rangle & \langle 348,368,356 \rangle \\
4 & \langle 38 \rangle & \langle 368,356 \rangle \\
5 & \langle 36 \rangle & \langle 368 \rangle
\end{array}
\]
Each link is a path, hence a tree, and we are left with the single facet $368$.
\end{example}

In this subsection we delete the elements of a class one at a time, which produces inflations in which one class is smaller than $d-1$, and the links that arise are inflations of cones. For a family of classes $\C$ and an element $p \in C_\ell$, we write $\C \setminus p$ for the family obtained from $\C$ by replacing $C_\ell$ with $C_\ell \setminus \{p\}$. We use without further comment that a skeleton of a simplex, and hence $\mE_t$ of a simplex, is vertex decomposable \cite{pb80}, and that an inflation $\mE^{\C}_t(\Gamma)$ is pure of dimension $t$ whenever the classes of every facet of $\Gamma$ have at least $t+1$ elements in total, since every face then extends to a face of cardinality $t+1$ with the same support.

\begin{lemma}
\label{lem:classvertex}
Let $t \geq 1$ and let $p \in C_\ell$ be a vertex of $\mE^{\C}_t(\Gamma)$. Then
\[
\lk_p{\mE^{\C}_t(\Gamma)} = \mE^{\C \setminus p}_{t-1}(\lk_\ell{\Gamma} * \ell) \qquad\text{and}\qquad \del_p{\mE^{\C}_t(\Gamma)} = \mE^{\C \setminus p}_t(\Gamma).
\]
\end{lemma}
\begin{proof}
A set $F$ with $p \notin F$ lies in $\lk_p{\mE^{\C}_t(\Gamma)}$ if and only if $|F| \leq t$ and $\supp(F) \cup \{\ell\} \in \Gamma$, that is, $\supp(F) \in \lk_\ell{\Gamma} * \ell$. This gives the formula for the link. The faces of $\mE^{\C}_t(\Gamma)$ avoiding $p$ are those with $F \cap C_\ell \subseteq C_\ell \setminus \{p\}$, which gives the deletion.
\end{proof}

Recall that we identify a graph with the simplicial complex whose facets are its edges and its isolated vertices.

\begin{lemma}
\label{lem:treecone}
Let $\Delta$ be a tree with at least two vertices and let $\C$ be a family of classes for the cone $\Delta * a$ in which every vertex of $\Delta$ has a class of cardinality $d-1$ and the apex has an arbitrary class $A := C_a$. Then $\mE^{\C}_{d-1}(\Delta * a)$ is vertex decomposable.
\end{lemma}
\begin{proof}
We induct on the number of vertices of $\Delta$. If $\Delta$ is a single edge $uw$, then $\Delta * a$ is a simplex and $\mE^{\C}_{d-1}(\Delta * a) = \skel{d-1}(2^{C_u \cup C_w \cup A})$. Otherwise let $\ell$ be a leaf with neighbor $w$. We delete the elements of $C_\ell$ one at a time. So let $\C'$ be $\C$ with $C_\ell$ replaced by a nonempty subset $L$, let $p \in L$, and let $\mE' := \mE^{\C'}_{d-1}(\Delta * a)$, which is pure of dimension $d-1$ as every facet of $\Delta * a$ contains a vertex of $\Delta$ other than $\ell$. By \cref{lem:classvertex}, $\del_p{\mE'} = \mE^{\C' \setminus p}_{d-1}(\Delta * a)$, and when $L = \{p\}$ this is $\mE^{\C}_{d-1}(\del_\ell{\Delta} * a)$, which is vertex decomposable by induction. By \cref{lem:classvertex} and $\lk_\ell(\Delta * a) = \langle wa \rangle$, the link $\lk_p{\mE'} = \mE^{\C' \setminus p}_{d-2}(\langle wa\ell \rangle) = \skel{d-2}(2^{C_w \cup (L \setminus \{p\}) \cup A})$ is vertex decomposable.

It remains to check that $p$ is gluable in $\mE'$. Let $F$ be a facet of $\mE'$ containing $p$. As $\ell$ is a leaf, $F \subseteq C_w \cup L \cup A$. If some $x \in C_w \cup L \cup A$ is not in $F$, then $(F \setminus \{p\}) \cup \{x\}$ is a facet of $\mE'$ avoiding $p$. Otherwise $|C_w \cup L \cup A| = d$, which forces $A = \emptyset$, $L = \{p\}$ and $F = C_w \cup \{p\}$, and then $C_w \cup \{x\}$ for any $x \in C_v$ is such a facet, where $v$ is a neighbor of $w$ other than $\ell$.
\end{proof}

\begin{lemma}[Vertex bundle]
\label{lem:vertexbundle}
Let $\Gamma$ be a pure $2$-dimensional complex and $a$ a gluable vertex of $\Gamma$ whose link $\Delta := \lk_a{\Gamma}$ is a tree. Let $\C$ be a family of classes for $\Gamma$ in which $C_a$ is nonempty and every other class has cardinality $d-1$. Then every $p \in C_a$ is a gluable vertex of $\mE^{\C}_d(\Gamma)$ with vertex decomposable link, and $\del_p{\mE^{\C}_d(\Gamma)} = \mE^{\C \setminus p}_d(\Gamma)$. In particular, deleting the elements of $C_a$ one at a time, in any order, takes $\mE_d(\Gamma)$ to $\mE_d(\del_a{\Gamma})$ through gluable vertices with vertex decomposable links.
\end{lemma}
\begin{proof}
The complex $\mE^{\C}_d(\Gamma)$ is pure of dimension $d$, as every facet of $\Gamma$ contains two vertices other than $a$. By \cref{lem:classvertex}, the link $\lk_p{\mE^{\C}_d(\Gamma)} = \mE^{\C \setminus p}_{d-1}(\Delta * a)$ is vertex decomposable by \cref{lem:treecone} (the tree $\Delta$ has at least two vertices as $\Gamma$ is pure of dimension $2$), and $\del_p{\mE^{\C}_d(\Gamma)} = \mE^{\C \setminus p}_d(\Gamma)$, which equals $\mE_d(\del_a{\Gamma})$ once $C_a = \{p\}$. For gluability, let $F$ be a facet of $\mE^{\C}_d(\Gamma)$ containing $p$ and $\sigma$ a facet of $\Gamma$ containing $\supp(F)$. The set $\bigcup_{v \in \sigma} C_v \setminus \{p\}$ has cardinality at least $2(d-1) \geq d+1$, so it contains some $x \notin F$, and $(F \setminus \{p\}) \cup \{x\}$ is a facet of $\mE^{\C}_d(\Gamma)$ avoiding $p$.
\end{proof}

\begin{lemma}[Edge bundle]
\label{lem:edgebundle}
Let $\Gamma$ be a pure $2$-dimensional complex and $ab$ a gluable edge of $\Gamma$ contained in exactly one facet $abc$. Let $B$ be the set of edges $pq$ of $\mE_d(\Gamma)$ with $p \in C_a$ and $q \in C_b$. For $P \subseteq B$, let $\Delta_P$ be the complex of faces of $\mE_d(\Gamma)$ containing no edge of $P$. Then every edge $pq \in B \setminus P$ is a gluable face of $\Delta_P$ with vertex decomposable link, and $\del_{pq}{\Delta_P} = \Delta_{P \cup \{pq\}}$. In particular, deleting the $(d-1)^2$ edges of $B$ one at a time, in any order, takes $\mE_d(\Gamma)$ to $\mE_d(\del_{ab}{\Gamma})$ through gluable edges with vertex decomposable links.
\end{lemma}
\begin{proof}
We show by induction on $|P|$ that $\Delta_P$ is pure of dimension $d$, which holds for $\Delta_\emptyset = \mE_d(\Gamma)$. So assume that $\Delta_P$ is pure of dimension $d$ and let $pq \in B \setminus P$. The formula $\del_{pq}{\Delta_P} = \Delta_{P \cup \{pq\}}$ is clear from the definition.

We first check that $pq$ is gluable. Let $F \cup \{p,q\}$ be a facet of $\Delta_P$, so that $|F| = d-1$ and $\supp(F) \subseteq \{a,b,c\}$, as $abc$ is the only facet of $\Gamma$ containing $ab$. We remove $p$, the case of $q$ being symmetric. If some $x \in C_c$ is not in $F$, then $F \cup \{q,x\}$ is a face of $\mE_d(\Gamma)$, and every edge of $B$ that it contains is already contained in $F \cup \{p,q\}$, so $F \cup \{q,x\}$ is a facet of $\Delta_P$ avoiding $p$. Otherwise $F = C_c$, and $C_c \cup \{q,q'\}$ for any $q' \in C_b \setminus \{q\}$ is a facet of $\Delta_P$ avoiding $p$. Hence $pq$ is gluable and $\Delta_{P \cup \{pq\}}$ is pure of dimension $d$, which completes the induction.

Next we look at the link. A set $F$ disjoint from $\{p,q\}$ lies in $\lk_{pq}{\Delta_P}$ if and only if $|F| \leq d-1$, $\supp(F) \subseteq \{a,b,c\}$, and $F \cup \{p,q\}$ contains no edge of $P$. Hence $\lk_{pq}{\Delta_P} = \skel{d-2}(K * C_c)$, where $K$ is the complex of subsets $U$ of $(C_a \setminus \{p\}) \cup (C_b \setminus \{q\})$ such that $U \cup \{p,q\}$ contains no edge of $P$. This is a pure $(d-2)$-dimensional complex coned with respect to $C_c$, and $|C_c| = d-1 = (d-2)+1$, so it is vertex decomposable by \cref{lem:conedfullVD} with $k = -1$, whatever $K$ is.

Finally, once every edge of $B$ is deleted we are left with the faces $F$ of $\mE_d(\Gamma)$ with $\{a,b\} \not\subseteq \supp(F)$, that is, with $\mE_d(\del_{ab}{\Gamma})$.
\end{proof}

\begin{proposition}
\label{prop:inflate1dec}
If $\Gamma$ is simply $1$-decomposable, then $\mE_d(\Gamma)$ is $1$-decomposable for every $d \geq 3$. In particular, $\mE_d(\Gamma_0)$ is $1$-decomposable for every $d \geq 3$.
\end{proposition}
\begin{proof}
We induct on the number of facets of $\Gamma$. If $\Gamma$ is a simplex $abc$, then $\mE_d(\Gamma) = \skel{d}(2^{C_a \cup C_b \cup C_c})$ is vertex decomposable. Otherwise let $F$ be a shedding face as in \cref{def:simply}. \Cref{lem:vertexbundle,lem:edgebundle} provide a sequence of gluable faces of dimension at most $1$ with vertex decomposable links that takes $\mE_d(\Gamma)$ to $\mE_d(\del_F{\Gamma})$. The latter is $1$-decomposable by induction, and going back along the sequence, each face is then a shedding face of a $1$-decomposable complex. The last statement follows from \cref{ex:gamma0simply}.
\end{proof}

The proof is constructive and produces an explicit $1$-decomposition of $\mE_d(\Gamma_0)$ from the twelve steps of \cref{ex:gamma0simply}.

\subsection{Tightness}
\label{subsec:tight}

\begin{theorem}
\label{thm:tight}
Let $d \geq 3$ and let $\mE := \mE_d(\Gamma_0)$ with vertex set $V$ of cardinality $8(d-1)$. Let $A$ be a set of at most $d-3$ new vertices. Then there is no ordering $F_1,\dots,F_t$ of the facets of $\Sk{d}{V \sqcup A} \setminus \mE$ such that $\mE + \langle F_1,\dots,F_i \rangle$ is shellable for all $1 \leq i \leq t$. In particular, for every $k \geq 1$, the complex $\Sk{d}{n+d-2}$ in \cref{thm:main} cannot be replaced by $\Sk{d}{n+d-3}$.
\end{theorem}
\begin{proof}
Let $\Delta' := \skel{d}(\mE * A)$ be the coning of $\mE$ by $A$. Since $|A| \leq d-3$, every facet $F$ of $\Sk{d}{V \sqcup A}$ has $|F \cap V| \geq 4$, and $F \in \Delta'$ if and only if $F \cap V \in \mE$.

We claim that if $\Delta$ is a pure complex with $\mE \subseteq \Delta \subseteq \Delta'$ and $F$ is a facet of $\Sk{d}{V \sqcup A}$ not in $\Delta'$, then $\Delta + F$ is not shellable. Indeed, $T := F \cap V$ is not a face of $\mE$ and $4 \leq |T| \leq d+1$, so by \cref{lem:inflate} $T$ contains two distinct minimal nonfaces $N_1,N_2$ of $\mE$. The faces of $\Delta'$ contained in $V$ are the faces of $\mE$, so $N_1,N_2$ are minimal nonfaces of $\Delta$, and \cref{lem:twononfaces} gives the claim.

Assume for the sake of contradiction that an ordering $F_1,\dots,F_t$ as in the statement exists. Not every facet of $\Sk{d}{V \sqcup A}$ lies in $\Delta'$, since $\Gamma_0$ is not $\Sk{2}{8}$, so let $F_i$ be the first facet in the ordering that is not in $\Delta'$. Then $\Delta := \mE + \langle F_1,\dots,F_{i-1} \rangle$ lies between $\mE$ and $\Delta'$, and $\Delta + F_i$ is not shellable by the claim proven in the previous paragraph, a contradiction. The last statement follows since $\mE$ is $1$-decomposable by \cref{prop:inflate1dec}.
\end{proof}

For $d = 3$ no new vertices are allowed, and \cref{thm:tight} says that the $280$-facet complex of \cite{bs26} is a $1$-decomposable complex on $16$ vertices that cannot be extended to $\Sk{3}{16}$ one facet at a time even while only maintaining shellability. This answers Question 5.1 of \cite{ghkklotw25} in the negative.

\section*{Acknowledgments}
The author thanks the participants of the 2021 REU hosted at Texas State University, Russell Barnes, Anton Dochtermann, Fran Herr, Cece Henderson and Ethan Partida, for the discussions that led to the idea of coning used in \cite{ghkklotw25} and in this paper. The author also thanks Benjamin Keller and Ryan Tang of the 2025 Honors Summer Math Camp group for the discussions that led to the idea of introducing new vertices, also used in \cite{ghkklotw25}. Finally, the author thanks the 2025 Honors Summer Math Camp at Texas State University for hosting, where many of the ideas in this paper originated.

\section*{Use of AI tools}
The author used two AI tools in preparing this paper: Claude (Anthropic, mainly the model Fable 5.1) and ChatGPT (OpenAI, mainly the model GPT-6 Astra). Below the author describes what each contributed, including to the mathematical reasoning.

AI contributed to the mathematical reasoning in \cref{subsec:lift}. ChatGPT found an explicit $1$-decomposition of the $280$-facet counterexample of Bolognini and Sentinelli \cite{bs26} by computer. This certificate and the code verifying it are available at \cite{o26code}, and are not needed for the results of this paper, since \cref{prop:inflate1dec} is proved by hand. Upon the author's request to use this counterexample to check whether the bound of $d-2$ new vertices in \cite{ghkklotw25} is tight, ChatGPT found the arguments of \cref{subsec:lift}, including the twelve-step decomposition of $\Gamma_0$ in \cref{ex:gamma0simply}.

All statements and proofs outside \cref{subsec:lift} are the author's. Claude read the proofs and pointed out gaps and missing cases, which the author closed, after which Claude was used to optimize and reword them. Claude drafted the manuscript, including the examples. The whole text was revised, edited and checked line by line by the author.

The author provided the main ideas of this paper: the coning operation and the notion of a $k$-full complex, the strategy of adding triangles to the base of a coned complex and finishing it off with vertex decomposability, and the idea of using the Bolognini--Sentinelli counterexample to show that the bound of $d-2$ new vertices from \cite{ghkklotw25} is tight.

The author checked all mathematical content and all computational results, and takes full responsibility for the paper.

\printbibliography

@article{bpz19,
 author = {Bigdeli, Mina and Yazdan Pour, Ali Akbar and Zaare-Nahandi, Rashid},
 title = {Decomposable clutters and a generalization of {Simon}'s conjecture},
 journal = {J. Algebra},
 volume = {531},
 pages = {102--124},
 year = {2019},
 doi = {10.1016/j.jalgebra.2019.03.037}
}

@article{bw96,
 author = {Bj{\"o}rner, Anders and Wachs, Michelle L.},
 title = {Shellable nonpure complexes and posets. {I}},
 journal = {Trans. Amer. Math. Soc.},
 volume = {348},
 number = {4},
 pages = {1299--1327},
 year = {1996},
 doi = {10.1090/S0002-9947-96-01534-6}
}

@article{be94,
 author = {Bj{\"o}rner, Anders and Eriksson, Kimmo},
 title = {Extendable shellability for rank 3 matroid complexes},
 journal = {Discrete Math.},
 volume = {132},
 number = {1-3},
 pages = {373--376},
 year = {1994},
 doi = {10.1016/0012-365X(94)90246-1}
}

@article{bm72,
 author = {Bruggesser, H. and Mani, P.},
 title = {Shellable decompositions of cells and spheres},
 journal = {Math. Scand.},
 volume = {29},
 pages = {197--205},
 year = {1972},
 doi = {10.7146/math.scand.a-11045}
}

@article{cdgo22,
 author = {Coleman, Michaela and Dochtermann, Anton and Geist, Nathan and Oh, Suho},
 title = {Completing and extending shellings of vertex decomposable complexes},
 journal = {SIAM J. Discrete Math.},
 volume = {36},
 number = {2},
 pages = {1291--1305},
 year = {2022},
 doi = {10.1137/21M1445119}
}

@article{cdgs20,
 author = {Culbertson, Jared and Dochtermann, Anton and Guralnik, Dan P. and Stiller, Peter F.},
 title = {Extendable shellability for $d$-dimensional complexes on $d+3$ vertices},
 journal = {Electron. J. Combin.},
 volume = {27},
 number = {3},
 pages = {Paper No. 3.46, 8},
 year = {2020},
 doi = {10.37236/9120}
}

@article{dk78,
 author = {Danaraj, Gopal and Klee, Victor},
 title = {Which spheres are shellable?},
 journal = {Ann. Discrete Math.},
 volume = {2},
 pages = {33--52},
 year = {1978},
 doi = {10.1016/S0167-5060(08)70320-0}
}

@article{d21,
 author = {Dochtermann, Anton},
 title = {Exposed circuits, linear quotients, and chordal clutters},
 journal = {J. Combin. Theory Ser. A},
 volume = {177},
 pages = {Paper No. 105327, 22},
 year = {2021},
 doi = {10.1016/j.jcta.2020.105327}
}

@misc{k77,
  title={Untersuchungen zur Struktur geometrischer Zellkomplexe insbesondere zur Schalbarkeit von pl-Sph{\"a}ren und pl-Kugeln},
  author={Kleinschmidt, P.},
  howpublished={Habilitationsschrift, Bochum},
  year={1977}
}

@article{mt03,
 author = {Moriyama, Sonoko and Takeuchi, Fumihiko},
 title = {Incremental construction properties in dimension two---shellability, extendable shellability and vertex decomposability},
 journal = {Discrete Math.},
 volume = {263},
 number = {1-3},
 pages = {295--296},
 year = {2003},
 doi = {10.1016/S0012-365X(02)00771-9}
}

@article{pb80,
 author = {Provan, J. Scott and Billera, Louis J.},
 title = {Decompositions of simplicial complexes related to diameters of convex polyhedra},
 journal = {Math. Oper. Res.},
 volume = {5},
 pages = {576--594},
 year = {1980},
 doi = {10.1287/moor.5.4.576}
}

@article{s94,
 author = {Simon, Robert Samuel},
 title = {Combinatorial properties of ``cleanness''},
 journal = {J. Algebra},
 volume = {167},
 number = {2},
 pages = {361--388},
 year = {1994},
 doi = {10.1006/jabr.1994.1191}
}

@article{w99,
 author = {Wachs, M. L.},
 title = {Obstructions to shellability},
 journal = {Discrete Comput. Geom.},
 volume = {22},
 number = {1},
 pages = {95--103},
 year = {1999},
 doi = {10.1007/PL00009450}
}

@article{z98,
 author = {Ziegler, G. M.},
 title = {Shelling polyhedral 3-balls and 4-polytopes},
 journal = {Discrete Comput. Geom.},
 volume = {19},
 number = {2},
 pages = {159--174},
 year = {1998},
 doi = {10.1007/PL00009339}
}

@book{j08,
 author = {Jonsson, Jakob},
 title = {Simplicial complexes of graphs},
 series = {Lecture Notes in Mathematics},
 volume = {1928},
 year = {2008},
 publisher = {Springer, Berlin},
 doi = {10.1007/978-3-540-75859-4}
}

@incollection{bj92,
 author = {Bj{\"o}rner, Anders},
 title = {The homology and shellability of matroids and geometric lattices},
 booktitle = {Matroid applications},
 pages = {226--283},
 year = {1992},
 publisher = {Cambridge Univ. Press, Cambridge}
}

@phdthesis{Hachimori2000,
  author = {Hachimori, Masahiro},
  title  = {Combinatorics of constructible complexes},
  school = {University of Tokyo},
  year   = {2000}
}

@article{Woodroofe2011,
  author  = {Woodroofe, Russ},
  title   = {Chordal and sequentially {C}ohen--{M}acaulay clutters},
  journal = {Electron. J. Combin.},
  volume  = {18},
  number  = {1},
  year    = {2011},
  pages   = {Paper No. 208, 20},
  doi     = {10.37236/695}
}

@article{AdiprasitoBenedettLutz2017,
  author  = {Adiprasito, Karim A. and Benedetti, Bruno and Lutz, Frank H.},
  title   = {Extremal examples of collapsible complexes and random discrete {M}orse theory},
  journal = {Discrete Comput. Geom.},
  volume  = {57},
  number  = {4},
  year    = {2017},
  pages   = {824--853},
  doi     = {10.1007/s00454-017-9860-4}
}

@book{Stanley96,
  author    = {Stanley, Richard P.},
  title     = {Combinatorics and Commutative Algebra},
  series    = {Progress in Mathematics},
  volume    = {41},
  edition   = {2nd},
  publisher = {Birkh{\"a}user Boston},
  year      = {1996}
}

@misc{bs26,
  author       = {Bolognini, Davide and Sentinelli, Paolo},
  title        = {Non-extendably shellable skeleta of simplices},
  year         = {2026},
  howpublished = {Preprint, arXiv:2605.24732},
  note         = {\url{https://arxiv.org/abs/2605.24732}}
}

@article{bb21,
 author = {Benedetti, Bruno and Bolognini, Davide},
 title = {Non-ridge-chordal complexes whose clique complex has shellable {Alexander} dual},
 journal = {J. Combin. Theory Ser. A},
 volume = {180},
 pages = {Paper No. 105430, 9},
 year = {2021},
 doi = {10.1016/j.jcta.2021.105430}
}

@article{dgl22,
 author = {Doolittle, Joseph and Goeckner, Bennet and Lazar, Alexander},
 title = {Partition and {Cohen--Macaulay} extenders},
 journal = {European J. Combin.},
 volume = {102},
 pages = {Paper No. 103488},
 year = {2022},
 doi = {10.1016/j.ejc.2021.103488}
}

@article{f25,
 author = {Ficarra, Antonino},
 title = {Simon conjecture and the v-number of monomial ideals},
 journal = {Collect. Math.},
 volume = {76},
 number = {3},
 pages = {477--492},
 year = {2025},
 doi = {10.1007/s13348-024-00441-z}
}

@misc{ghkklotw25,
  author       = {Ghosal, Rhea and Han, Melody and Keller, Benjamin and Kerr, Scarlett and Liu, Justin and Oh, Suho and Tang, Ryan and Weng, Chloe},
  title        = {Extendability of $1$-decomposable complexes},
  year         = {2025},
  howpublished = {Preprint, arXiv:2508.04555},
  note         = {\url{https://arxiv.org/abs/2508.04555}}
}

@phdthesis{hall04,
  author = {Hall, H. Tracy},
  title  = {Counterexamples in Discrete Geometry},
  school = {University of California, Berkeley},
  year   = {2004}
}

@article{e96,
  author  = {Eriksson, Kimmo},
  title   = {Strong convergence and the polygon property of 1-player games},
  journal = {Discrete Math.},
  volume  = {153},
  number  = {1-3},
  pages   = {105--122},
  year    = {1996},
  doi     = {10.1016/0012-365X(95)00131-F}
}

@book{z95,
  author    = {Ziegler, G{\"u}nter M.},
  title     = {Lectures on Polytopes},
  series    = {Graduate Texts in Mathematics},
  volume    = {152},
  publisher = {Springer-Verlag},
  address   = {New York},
  year      = {1995}
}

@misc{o26code,
  author       = {Oh, Suho},
  title        = {Code for ``{C}ompleting shellings with $d-2$ extra vertices''},
  year         = {2026},
  howpublished = {Google Colab notebook},
  note         = {\url{https://colab.research.google.com/drive/15qYmoXz5DbeWIWaTSOylSCz_KdA7i6Vm}}
}

\end{document}